\documentclass[12pt,reqno]{amsart}
\usepackage[english]{babel} \usepackage{latexsym, amsfonts, amsmath, amssymb, amsthm, mathrsfs}
\usepackage{color,lmodern} \usepackage{mathpazo} \usepackage[pdftex]{hyperref}
\usepackage[top=2.6cm, bottom=2.6cm, left=2.6cm, right=2.6cm]{geometry}
\numberwithin{equation}{section}  \makeatletter\@addtoreset{equation}{section}
\newtheorem {theorem}{Theorem}[section]
 \newtheorem {lemma}[theorem]{Lemma}     
\newtheorem {corollary}[theorem]{Corollary}     \newtheorem {remark}[theorem]{Remark}   
       \newtheorem {proposition}[theorem]{Proposition}
\newcommand{\C}{\mathbb C}    \newcommand{\R}{\mathbb R}    	
\newcommand{\norm}[1]{\left\Vert#1\right\Vert}  \newcommand{\scal}[1]{\left<#1\right>}
\newcommand{\Hq}{\mathbb H}
\newcommand{\Sq}{\mathbb S}
\newcommand{\B}{\mathbb B}
\newcommand{\BR}{\mathbb{B}_{R}}
\newcommand{\BRI}{\mathbb{B}_{R,I}}
\newcommand{\alphaa}{\alpha}
\newcommand{\SHyperBTransR}{\mathcal{A}^\alphaa_{R,slice}}
\newcommand{\SHyperBKernelFctR}{A^\alphaa_{R,slice}}
\newcommand{\BergHol}{A_{hol}^{2,\alpha}(D,\C)}

\newcommand{\BergHolRI}{A^{2,\alpha}_{hol}(\BRI)}

\newcommand{\BergSliceR}{A^{2,\alpha}_{slice}(\BR )}
\newcommand{\HolKernelFctR}{A^\alphaa_{R,hol}}
\newcommand{\HolBTransR}{\mathcal{A}^\alpha_{R,hol}}
\newcommand{\HolBTrans}{\mathcal{A}^\alpha_{hol}}
\newcommand{\Kslice}{K^{\alphaa}_{R,slice}}
\newcommand{\Khol}{K^{\alphaa}_{R,hol}}

\begin{document}
 \dedicatory{ \textit{Dedicated to the memory of Kettani Ghanmi and Professor Ahmed Intissar}}

\title[The slice hyperholomorphic Bergman space on $\BR$]
{The slice hyperholomorphic Bergman space on $\BR$: Integral representation and asymptotic behavior}
\thanks{The research work of A.G. was partially supported by a grant from the Simons Foundation. 
}
\author{A. El Kachkouri}
\author{A. Ghanmi}
 \address{ \quad \newline
          Analysis and Spectral Geometry (A.G.S.), \newline
          Laboratory of Mathematical Analysis and Applications (L.A.M.A.) \newline
          Center of Mathematical Research and Applications of Rabat (CeReMAR), \newline
          Department of Mathematics, P.O. Box 1014,  Faculty of Sciences \newline
          Mohammed V University in Rabat, Morocco}
           \email{elkachkouri.abdelatif@gmail.com} \email{ag@fsr.ac.ma}

\subjclass[2010]{Primary 30G35} 
\keywords{Slice regular functions; Slice hyperholomorphic Bergman space;
 Second Bargmann transform; Laguerre polynomials; Asymptotic behavior;  Slice hyperholomorphic Bargmann-Fock space}

\begin{abstract}
The aim of the present paper is three folds. Firstly, we complete the study of the weighted hyperholomorphic Bergman space of the second kind on the ball of radius $R$ centred at the origin. The explicit expression of its Bergman kernel is given and can be written in terms of special hypergeometric functions of two non-commuting (quaternionic) variables. Secondly, we introduce and study some basic properties of an associated integral transform, the quaternionic analogue of the so-called second Bargmann transform for the holomorphic Bergman space.
Finally, we establish the asymptotic behavior as $R$ goes to infinity. We show in particular that the reproducing kernel of the weighted slice hyperholomorphic Bergman space gives rise to its analogue for the slice hyperholomorphic Bargamann-Fock space.
\end{abstract}
\maketitle

\section{\bf Introduction}

The classical Segal-Bargmann transform is well-known in the literature \cite{Bargmann1961,Folland1989,Zhu2012} and made the quantum mechanical configuration space $L^{2}(\R,\C)$ unitarily isomorphic to the phase space of all holomorphic $\C$-valued functions on the complex plane that are $e^{-\nu  |z|^2}dxdy$-square integrable. Added to this transform, V. Bargmann has introduced in the same paper \cite[p.203]{Bargmann1961} the integral operator \begin{align}\label{sBT}
\HolBTrans \varphi(z) = (1-z)^{-\alpha-1} \int_0^{+\infty}  \varphi(t) \exp \left( \frac{tz}{z-1}\right) \frac{t^{\alpha}e^{-t}}{\Gamma (\alpha+1)}   dt
\end{align}
mapping isometrically the Hilbert space
$$L^{2,\alpha}(\R^+,\C):=L^2\left(\R^+; \frac{t^{\alpha}e^{-t}}{\Gamma (\alpha+1)} dt\right) ; \quad \alpha>0,$$
onto the classical holomorphic Bergman space $\BergHol= Hol(D,\C) \cap L^{2,\alpha}( D, \C)$
consisting of all $\C$-valued holomorphic functions on the unit disk $D=\{z\in \C; \, |z|<1 \}$ that are square integrable with respect to the hyperbolic measure
\begin{align}\label{w3}
 d\lambda_\alpha(z) := (1-|z|^2)^{\alpha-1}\frac{\alpha}{\pi}   dxdy .
 \end{align}
The transform in \eqref{sBT} is realized as a coherent state transform associated to the lower hyperbolic Landau level of a special magnetic Schr\"odinger operator on the Poincar\'e disk \cite{EGIM2012}. In fact, the involved kernel function is related to the generating function of the Laguerre polynomials $L^{(\alpha)}_n$. A $q$-analogue of the weighted Bergman Hilbert space and the corresponding integral transform of Bargmann type in the setting of the $q$-analysis are considered in \cite{EssadiqGI2016}.

The theory of slice regularity initiated by Gentili and Struppa in \cite{GentiliStruppa07} extends in an appropriate way the holomorphic setting on $\C$ to $\Hq$-valued functions of one quaternoinic variable.
It was extensively studied and has found many interesting applications in operator theory, quantum physics and Schur analysis \cite{ColomboSabadiniStruppa2011,AlpayColomboSabadini2012,AlpayColomboSabadini2013,GentiliStoppatoStruppa2013,
AlpayBolotnikovColomboSabadini2016,AlpayColomboSabadini2016}.
In \cite{AlpayColomboSabadini2014}, Alpay {\it et al.} have considered the slice hyperholomorphic Bargmann-Fock space
  \begin{align}\label{SliceBargmann}
   \mathcal{F}^{2,\nu}_{slice}(\Hq) = \mathcal{SR}(\Hq) \cap L^{2,\nu}(\C_I,\Hq),
  \end{align}
where $\mathcal{SR}(\Hq)$ denotes the space of (left) slice regular $\Hq$-valued functions on quaternion and
$L^{2,\nu}(\C_I,\Hq)$, $\nu>0$, is the Hilbert space of $\Hq$-valued functions that are square integrable with respect to the Gaussian measure on an arbitrary slice $\C_I=\R+ \R I$. The corresponding Segal-Bargmann transform is considered in \cite{DG2017}
and maps isometrically the $L^2$-Hilbert space $L^{2}(\R,\Hq)$ onto the slice hyperholomorphic Bargmann-Fock space
  $ \mathcal{F}^{2,\nu}_{slice}(\Hq)$.
A quaternionic analogue of the Bergman theory in the setting of the slice regular functions on the open unit ball (centered at the origin) has been introduced by Colombo {\it et al.} in \cite{Colombo-Sabadini} (see also \cite{ColomboGonzalez-CervantesSabadini2014}).
Thus, the slice hyperholomorphic Bergman space $\BergSliceR$, for arbitrary radius, is defined to be
\begin{align}\label{SliceBergman}
\BergSliceR :=   \mathcal{SR}(\BR ) \cap L^{2,\alpha}(\BRI,\Hq) , \quad \alpha>0,
\end{align}
where $L^{2,\alpha}(\BRI,\Hq)$ is the $L^2$-Hilbert space on $\BRI=\BR \cap\mathbb{C}_{I}$ of $\Hq$-valued functions $f$ subject to the norm boundedness
\begin{align}\label{SP-slice}
\norm{ f}_{\alpha,slice}^2 =\int_{\BRI}\overline{f(z)}g(z)
 \left(1-\frac{|z|^2}{R^2}\right)^{\alpha-1} \frac{\alpha}{\pi R^2} dxdy
< +\infty.
\end{align}
Notice that the parameter $\alpha$ above corresponds to $\alpha+1$ in \cite{Colombo-Sabadini,ColomboGonzalez-CervantesSabadini2014} and the measure in \eqref{SP-slice} is the volume measure associated to the quaternionic hyperbolic geometry on $\BR$ induced from the scaled Poincar\'e-type differential metric $ ds^2_{\BR} =  R^4 (R ^2-|q|^2)^{-2} |d_Iq|^2   $.
The metric $ds^2_{\BR}$ (with $R=1$) was defined in \cite{BisiGentili2009} by developing a variation of an approach adopted by Ahlfors \cite{Ahlfors1988}.

 Motivated by these recent investigations in the theory of slice regularity and the geometrical fact that the euclidean limit of the balls $\BR$ (hyperbolic case) gives rise to the quaternionic space $\Hq$ (flat case), as the radius $R$ goes to $+\infty$,
 quite natural questions arise of whether the analogue of the transform $\HolBTrans$ can be constructed and wether the two theories on $ \mathcal{F}^{2,\nu}_{slice}(\Hq)$ and $\BergSliceR$ can be connected.
Our main purpose in the present paper is to answer these two questions.
Namely, we establish a quaternionic analogue of \eqref{sBT} for the slice hyperholomorphic Bergman space (Theorem \ref{MainThm1}). Moreover, related basic properties are studied and the explicit expression of its inverse is obtained (Theorem \ref{MainThm2}).
We also exhibit an orthonormal basis (Proposition \ref{prop:normBergman}) and give in Theorem \ref{thm:explicitWBK} the closed expression of its reproducing kernel generalizing the one obtained in \cite[Theorem 4.1-Proposition 4.3]{ColomboGonzalez-CervantesSabadini2014} for $\alpha=1$. We also provide an
integral representation (Theorem \ref{thm:Kernel-Kernel}) of this reproducing kernel.
This integral representation involves the kernel function of the quaternionic analogue of the second Bargmann transform for which we present in Theorem \ref{KerFctExplicit} a closed form of its explicit formula.
The study will be done on the quaternionic ball of radius $R$ centred at the origin so that the asymptotic behavior as $R$ goes to infinity can be discussed.
We show in Theorem \ref{MainThm3:limKernel} that the pointwise limit of the weighted Bergman kernel of the slice hyperholomorphic Bergman space $\BergSliceR$, for the specific $\alpha = \nu R^2$, is exactly the reproducing kernel of the slice hyperholomorphic Bargmann-Fock space $ \mathcal{F}^{2,\nu}_{slice}(\Hq)$. This is to say that one can move from the Bergman universe to the Bargmann universe by taking the "euclidean limit".

The following structure is adopted. To make the paper self-contained, we review in Section 2 some basic mathematical concepts relevant to slice regular functions. For more details, we refer the reader to \cite{ColomboSabadiniStruppa2011,GentiliStoppatoStruppa2013,CP} and the references therein.
In Section 3 we complete the study of basic properties of the slice hyperholomorphic Bergman space. Section 4 is devoted to the exact statements and the proofs of our main results concerning the quaternionic analogue of the second Bargmann transform. The last section discusses the asymptotic behavior, as $R$ goes to $+\infty$, of some elements in the slice heperholomorphic Bergman theory, like measure, basis and the reproducing kernel function. It will be seen that they give rise to their analogues in the setting of slice hyperbolic Bargmann-Fock space.

\section{\bf \textbf{Preliminaries }}

 Let $\Hq$ denote the quaternion algebra with its standard basis $\lbrace{1,i,j,k}\rbrace$ satisfying the Hamiltonian multiplication $i^2=j^2=k^2=ijk=-1$, $ij=-ji=k$, $jk=-kj=i$ and $ki=-ik=j$.  For $q\in{\Hq}$, we write $q=x_0+x_1i+x_2j+x_3k$ with $x_0,x_1,x_2,x_3\in{\R}$. With respect to the quaternionic conjugate defined to be $\overline{q}=x_0-x_1i-x_2j-x_3k=\Re(q)-\Im(q)$, we have
 $\overline{ pq }= \overline{q}\, \overline{p}$ for $p,q\in \Hq$. The modulus of $q$ is defined by $\vert{q}\vert=\sqrt{q\overline{q}}=(x_0^2+x_1^2+x_2^2+x_3^2)^{1/2}$.
 The unit sphere $S^2=\lbrace{q\in{\Im\Hq}; \vert{\Im(q)}\vert=1}\rbrace$ in $\Im\Hq$ can be identified with $\mathbb{S}=\lbrace{q\in{\Hq};q^2=-1}\rbrace$, the set of imaginary units.
  Moreover, any $q\in \Hq\setminus \R$ can be rewritten in a unique way as $q=x+I y$ for some real numbers $x$ and $y>0$, and imaginary unit $I\in \mathbb{S}$.
  Accordingly, $\Hq$ is the union of the so-called slices, $\C_I = \R+\R I$; $I\in{\mathbb{S}}$, that are complex planes in $\Hq$ (passing through $0$, $1$ and $I$) isomorphic to the complex plane $\C$.

The basic notion in this section is the slice (left) regularity (or hyperolomorphicity) of a function $f: \Omega \longrightarrow \Hq$ on a given domain $\Omega\subset \Hq$, provided that $f$ is a real differentiable function on $\Omega$ and its restriction $f_I$ is holomorphic on $\Omega_I := \Omega \cap \C_I$. That is, it has continuous partial derivatives with respect to $x$ and $y$ and the function
$\overline{\partial_I} f : \Omega_I \longrightarrow \Hq$ defined by
$$
\overline{\partial_I} f(x+Iy):=
\dfrac{1}{2}\left(\frac{\partial }{\partial x}+I\frac{\partial }{\partial y}\right)f_I(x+yI)
$$
vanishes identically on $\Omega_I$. The corresponding space, denoted $\mathcal{SR}(\Omega)$, is endowed with the natural uniform convergence on compact sets. It turns out that $\mathcal{SR}(\Omega)$ is a right vector space, over the noncommutative field $\Hq$.
For $\Omega$ being the open ball  $\Omega = B(0,R):= \{q\in \Hq; \, |q| < R\}$, a function $f$ is in $\mathcal{SR}(B(0,R))$ if and only if it admits a power series expansion (\cite{GentiliStruppa07,ColomboSabadiniStruppa2011,GentiliStoppatoStruppa2013})
\begin{align}\label{expansion}
f(q)=\sum_{n=0}^{+\infty} q^{n} a_n; \qquad a_n = \frac{1}{n!}\frac{\partial^{n}f}{\partial x^{n}}(0)
\end{align}
which converges absolutely and uniformly on every compact subset of $B(0,R)$.
Interesting results for the slice hyperholomorphic functions are stated in the context of the whole space $\Hq$ as well as of the Euclidean ball $\BR =B(0,R)$.
These two domains are special examples of the so-called axially symmetric slice domains.
A slice domain $\Omega\subset \Hq$ is such that $\Omega\cap{\R}\ne \emptyset$ and the set $\Omega_I:=\Omega\cap{\C_I}$ is a domain of the complex plane $\C_I$ for any arbitrary $I\in{\mathbb{S}}$, while the axial symmetry means that the whole sphere $x+y\mathbb{S}:=\lbrace{x+yJ; \, J\in{\mathbb{S}}}\rbrace$ is contained in $\Omega$ for every $q=x+yI\in{\Omega}$.

We conclude this section by recalling some fundamental results in the theory of slice regular functions that are of particular interest for us (see \cite{ColomboSabadiniStruppa2011,GentiliStoppatoStruppa2013} for details). The first one relates slice regularity to classical holomorphy. 

\begin{lemma}[Splitting lemma]\label{split} Let $f$ be a slice regular function on an open set $U$.  For every $I$ and $J$ two perpendicular imaginary units, there exist two holomorphic functions $F,G:U_I=U\cap{\C_I}\longrightarrow{\C_I}$ such that for all $z=x+yI\in U_I$, we have
$$f_I(z)=F(z)+G(z)J.$$
\end{lemma}

The global behavior of a slice holomorphic function on axially symmetric set is completely determined by their behavior on a given slice. More precisely, we have

\begin{lemma}[Representation formula]\label{repform}
Let $\Omega$ be an axially symmetric slice domain and $f\in{\mathcal{SR}(\Omega)}$. Then, for any $I,J\in{\mathbb{S}}$ and every  $q=x+yJ\in{\Omega}$, we have
$$ f(x+yJ)= \frac{1}{2}(1-JI)f(x+yI)+ \frac{1}{2}(1+JI)f(x-yI).$$
\end{lemma}

\begin{lemma}[Extension Lemma]\label{extensionLem}
Let $I\in \mathbb{S}$ and $h:\Omega_I\longrightarrow \Hq$ be a holomorphic function on a symmetric domain $\Omega_I=\Omega\cap \C_I$ in $\C_I$ with respect to the real axis. Then, the function $ext(h)$ defined by
$$ext(h)(x+yJ):= \dfrac{1}{2}[h(x+yI)+h(x-yI)]+\frac{JI}{2}[h(x-yI)-h(x+yI)];  \quad J\in \mathbb{S},$$
extends $h$ to a slice regular function on the symmetric completion of $\Omega_I$ defined by $\overset{\sim}\Omega=\cup \{x+y\mathbb{S}; x+yJ\in{\Omega}\}$.
Moreover, $ext(h)$ is the unique slice regular extension of $h$.
\end{lemma}

  \begin{lemma}[Identity principle]\label{IdentityPrinciple}
Let $f$ be a slice regular function on a slice domain $U$ and denote by $\mathcal{Z}_f$ its zero set.
If $\mathcal{Z}_f \cap \C_I$ has an accumulation point in $U_I$ for some $I\in \Sq$, then $f$ vanishes identically on $U$.
\end{lemma}

\section{\bf The slice hyperholomorphic Bergman space of the second kind revised}

The definition of the slice hyperholomorphic Bergman space 
 on the open unit ball $\mathbb{B}$ (centered at the origin) was first presented in \cite{Colombo-Sabadini}. This was possible by extending the complex holomorphic functions on the disc to the whole $\mathbb{B}$ by the representation formula (Lemma \ref{repform}). For arbitrary radius $R$, the slice hyperholomorphic Bergman space  $\BergSliceR$ is defined by \eqref{SliceBergman}, $\BergSliceR :=   \mathcal{SR}(\BR ) \cap L^{2,\alpha}(\BRI,\Hq)$.
The extra normalisation factor ${\alpha}/{\pi R^2}$ in \eqref{SP-slice} defining the measure is implemented to simplify later formulas and mainly to get the asymptotic behavior when $R$ goes to infinity. It turns out that
$\BergSliceR|_{\BRI} := \left\{ f_I; \,  f\in \BergSliceR \right\} $
is the usual Bergman space on the disc $\BRI=D_I(0,R)\subset \C_I$ with respect to the norm $\norm{\cdot}_{\alpha,slice}$
in \eqref{SP-slice}.
Moreover, it is shown in \cite[Theorem 3, p. 50]{Colombo-Sabadini} that $\BergSliceR $ is a reproducing kernel Hilbert space with respect to \eqref{SP-slice}, whose the reproducing kernel $K^{\alpha}_{R}(q, q')$ satisfies
$$f(q) =\int_{\BRI} \Kslice (q,z)f(z) 
\left(1-\frac{|z|^2}{R^2}\right)^{\alpha-1} \frac{\alpha}{\pi R^2} dxdy$$
for all $f \in \BergSliceR$ and every $I\in  \mathbb{S}^{2}$. Notice for instance that by the representation formula (Lemma \ref{repform}), the involved integral does not depend on the choice of $I \in \mathbb{S}^{2}$.
Moreover, it is clear that the restriction $ K^{\alpha}_{R,I}:=\Kslice|_{\BRI\times \BRI}$ of the Bergman kernel to $ \BRI\times \BRI$ coincides with the classical Bergman kernel $\Khol(z,w)$ on $\BRI$ given by
\begin{align}\label{expKhol}
\Khol(z,w) = \left( 1 - \frac{z\overline{w}}{R^2} \right)^{-\alpha-1} = K^{\alpha}_{R,I}(z,w).
\end{align}
 The explicit expression of $\Kslice(q,p)$ for $\alpha=1$ and $R=1$,
 is proved in \cite[Theorem 4.1-Proposition 4.3]{ColomboGonzalez-CervantesSabadini2014} to be given by
\begin{align}
\Kslice(q, p) &= (1 - 2\overline{q}\, \overline{p} + \overline{q}^2\overline{p}^2)(1 - 2\Re(q)\overline{p} + |q|^2\overline{p}^2)^{-2} \label{SliceBergmanKernel1}
\\  &=  (1 - 2q\Re(p) + q^2|p|^2)^{-2}(1 - 2qp + q^2p^2).
 \label{SliceBergmanKernel2}
\end{align}
A direct computation shows that the two expressions \eqref{SliceBergmanKernel1} and \eqref{SliceBergmanKernel2}
are the same.
For general $\alpha >0$, the expression of $\Kslice(q, p)$ can be given in terms of the special function
\begin{align}\label{I-HypergeometricFct}
 I^a(q,p) :=  \sum_{n=0}^\infty \frac{(a)_n }{n!} q^np^n ,
\end{align}
with real parameter $a$ and quaternionic variables $q,p\in \Hq$, which is a particular case of the left-sided Gauss hypergeometric function
\begin{equation}\label{NewHypergeometricFct}
 {_2{F^{*}}_1}\left( [q,p] \bigg |   \begin{array}{c} a , b \\ c \end{array} \right)
 = \sum_{n=0}^\infty  \frac{q^np^n}{n!} \frac{(a)_n (b)_n}{(c)_n}
\end{equation}
defined here for real $c$ and quaternionic $a,b\in \Hq$.
Above $(a)_k$ denotes the Pochhammer symbol $(a)_k= a(a+1) \cdots (a+k-1)$ with $(a)_0=1$.
The above series converges absolutely and uniformly on $K\times K'$ for any compact subsets $K,K'\subset \BR $.

\begin{theorem}\label{thm:explicitWBK}
The weighted Bergman kernel is given by
\begin{align}\label{explicitWBK}
\Kslice(q, p)
&= I^{-\alpha-1} \left(\frac{\overline{q}}{R} ,\frac{\overline{p}}{R}\right) \left( 1 - 2 \frac{\Re(q) \overline{p} }{R^2} +  \frac{|q|^2\overline{p}^2}{R^4} \right)^{-\alpha-1}
\end{align}
and
\begin{align}\label{explicitWBK2}
\Kslice(q, p)
= I^{\alpha+1}\left(\frac{\overline{q}}{R} ,\frac{\overline{p}}{R}\right) .
\end{align}
\end{theorem}

\begin{proof}
 Fix $I\in \Sq$ and $q\in \BRI$. The equality \eqref{explicitWBK} holds trivially for every $p\in \BRI$ since both sides of \eqref{explicitWBK} reduce further to
 $K^{\alpha}_{R,I}(q,p)$ in \eqref{expKhol}.
The assertion of Theorem \ref{thm:explicitWBK} for arbitrary $p\in \BR $ immediately follows from the identity principle (Lemma \ref{IdentityPrinciple}) for right anti-slice regular functions. Indeed, the function
$$ p \longmapsto \overline{\Kslice(q, p)} - I^{-\alpha-1} \left(\frac{p}{R} ,\frac{q}{R}\right)\left( 1 - 2 \frac{\Re(q) p }{R^2} +  \frac{|q|^2p^2}{R^4} \right)^{-\alpha-1}$$
vanishes on $\BRI$ and is left slice regular for the coefficients in the expansion series of $I^{-\alpha-1} $ being reals. Thus, it is identically zero on the whole $\BR$.
The proof of \eqref{explicitWBK2} can be handled using similar arguments based essentially on the counterpart of the identity principle for right slice regular functions.
This completes the proof.
\end{proof}

\begin{remark}
 The explicit expression of $\Kslice(q,p)$ in \eqref{explicitWBK} in terms of the special function $ I^{-\alpha-1} $ in \eqref{I-HypergeometricFct} can be suggested starting from \eqref{expKhol} and using the extension Lemma \ref{extensionLem} (see the appendix). Being indeed, we have
\begin{align}
\Kslice(x + yI,p)&= \frac{1-II_p}{2} \Khol(x + yI_{p},p) + \frac{1+II_p}{2} \Khol(x - yI_{p},p) .
\end{align}
The proof presented is more simpler.
\end{remark}

\begin{remark}
For $\alphaa$ being a nonnegative integer the expression \eqref{explicitWBK} reduces further to the following
\begin{align}
\Kslice(q, p) = P_{\alphaa+1}\left(\frac{\overline{q}}{R},\frac{\overline{p}}{R}\right) \left(1 - 2\frac{\Re(q)\overline{p}}{R^{2}} +\frac{ |q|^2\overline{p}^2}{R^{4}}\right)^{-\alphaa-1} ,\label{SliceBergmanKernel1Alpha}
\end{align}
where $P_{\alphaa+1}(\overline{q},\overline{p})$ is the polynomial of degree $\alphaa+1$ given by
$$ P_{\alphaa+1}(\overline{q},\overline{p}) = \sum_{k=0}^{\alphaa+1}     \frac{ (-\alphaa-1)_k }{k!}  \overline{q}^k \overline{p}^k   .$$
When taking $\alphaa=1$, we recover \eqref{SliceBergmanKernel1} obtained in \cite{ColomboGonzalez-CervantesSabadini2014}.
\end{remark}

\begin{corollary}\label{cor:expansionKernel}
We have the following identity (for the hypergeometric function ${_2{F^{*}}_1}$ in \eqref{NewHypergeometricFct}),
\begin{align}\label{Identity}
I^{\alpha+1}\left(\frac{q}{R} ,\frac{\overline{p}}{R}\right) =I^{-\alpha-1} \left(\frac{\overline{q}}{R} ,\frac{\overline{p}}{R}\right)\left( 1 - 2 \frac{\Re(q) \overline{p} }{R^2} +  \frac{|q|^2\overline{p}^2}{R^4} \right)^{-\alpha-1}.
\end{align}
In particular, for $\alphaa=1$, we have
\begin{align}\label{identiy2}
I^2\left(\frac{q}{R} ,\frac{\overline{p}}{R}\right)
= \left(1 - 2\frac{q\Re(p)}{R^2} + \frac{q^2|p|^2}{R^4}\right)^{-2}\left( 1 - 2\frac{qp}{R^2} + \frac{q^2p^2}{R^4}\right)
\end{align}
\end{corollary}

\begin{proof}
The identity \eqref{Identity} follows by equating the right hand sides in both \eqref{explicitWBK} and \eqref{explicitWBK2}. It can also be obtained
using similar arguments as in the proof of Theorem \ref{thm:explicitWBK}.
A direct proof outside the framework of slice regular functions seems to be hard to obtain for the lack of commutativity in the quaternions.
The second identity \eqref{identiy2} is a particular case keeping in mind the expression of $\Kslice(q, p)$, for arbitrary $R$, given through
\eqref{SliceBergmanKernel2}.
\end{proof}

\begin{remark}
The formula \eqref{explicitWBK2} and therefore the identity \eqref{Identity} can be reproved using Proposition \ref{prop:normBergman} below,
since the $\Kslice(q,p)$ can be realized as
\begin{align}\label{RepKerExpansian}
\Kslice(q,p)=  \sum_{n=0}^\infty \phi_n(q) \overline{\phi_n(p)}
\end{align}
for any orthonormal total family of functions $(\phi_n)_n$ in $\BergSliceR$.
The involved series converges uniformly on $K\times K$ for any compact subset $K\subset \BR $.
\end{remark}

\begin{proposition}\label{prop:normBergman}
The monomials $e_n(q):=q^{n}$ form an orthogonal basis of $\BergSliceR$ with respect to \eqref{SP-slice}. The square norm of the $e_n$ is given by
\begin{align}\label{norm}
\norm{e_n}_{\alpha,slice}^2  =\frac{n!R^{2n}\Gamma(\alpha+1)}{\Gamma(n+\alpha+1)}.
\end{align}
\end{proposition}

\begin{proof}
The first assertion follows by similar arguments as in the classical case. The norm of $e_n(q):=q^{n}=r^ne^{In\theta}$ can be computed easily using the polar coordinates and making use of the appropriate change of variable $t=r^2/R^2$. Indeed,
\begin{align*}
\scal{e_n,e_m}_{\alpha,slice}
&= \left(\frac{\alpha}{\pi R^2} \right)\int_{\BRI} \overline{q}^nq^m \left( 1 -\frac{|q|^2}{R^2} \right)^{\alpha-1} dxdy
\\&= \alpha R^{2n} \delta_{m,n} \int_0^1 t^{n} (1-t)^{\alpha-1} dt
\\&=\frac{n!R^{2n}\Gamma(\alpha+1)}{\Gamma(n+\alpha+1)} \delta_{m,n}.
\end{align*}
The last equality follows making use of $\alpha\Gamma(\alpha)=\Gamma(\alpha+1)$ as well as of the well-known Euler's Beta integral  \cite[Theorem 7, p. 19]{Rainville71}
$$ \int_0^1 t^{a-1} (1-t)^{b-1} dt  = \frac{\Gamma(a)\Gamma(b)}{\Gamma(a+b)}$$
valid for $\Re(a)>0$ and $\Re(b)>0$.
\end{proof}

As immediate consequence of Proposition \ref{prop:normBergman}, one can easily obtain the following
$$ \scal{f,g}_{\alpha,slice} = \sum_{n\geqslant0} \frac{n!R^{2n}\Gamma(\alpha+1)}{\Gamma(n+\alpha+1)} \overline{a_n}b_n  $$
for any $f(q) = \sum\limits_{n=0}^{\infty} q^{n}a_n$ and $g(q) = \sum\limits_{n=0}^{\infty} q^{n}b_n$ in  $\BergSliceR$.
In particular, we assert

\begin{corollary}\label{cor:growthCond}
A given $f(q) = \sum\limits_{n=0}^{\infty} q^{n}a_n$; $a_n\in\Hq$, belongs to $\BergSliceR$ if and only if the coefficients $a_n$ satisfies the growth condition
\begin{align}\label{growthCond}
\norm{f}^{2}_{\alpha,slice}=\sum_{n=0}^{\infty} \frac{n!R^{2n}\Gamma(\alpha+1)}{\Gamma(n+\alpha+1)} |a_{n}|^{2} < +\infty.
\end{align}
\end{corollary}

\begin{remark}
The identity \eqref{growthCond} shows in particular that the quantity $\norm{f}^{2}_{\alpha,slice}$ is independent of the choice of the purely imaginary unit $I\in\Sq$.
 \end{remark}

Now, let define $\SHyperBKernelFctR$ to be
\begin{equation} \label{KerFct2}
\SHyperBKernelFctR (t;x+yI_{q})
= \frac{1 - I_{q}J}{2} \HolKernelFctR (t;x + yJ) + \frac{1 + I_{q}J}{2} \HolKernelFctR (t;x - yJ)
\end{equation}
defined on $\R^{+} \times \BR $, where
\begin{equation} \label{cBergmanKernel}
\HolKernelFctR(t;z) := \exp \left( \frac{tz}{z-R}\right) \left(1-\frac{z}{R}\right)^{-\alphaa -1}
\end{equation}
is the kernel function of the second Bargmann transform \eqref{sBT}.
The following result shows that the sliced weighted Bergman kernel $\Kslice$ is connected to the kernel function
$\SHyperBKernelFctR$.

\begin{theorem}\label{thm:Kernel-Kernel}
For every $q,q'\in \BR $, we have
$$\int_{0}^{+\infty}\SHyperBKernelFctR(t;q)\overline{\SHyperBKernelFctR(t;q')} \frac{t^\alphaa  e^{-t}}{\Gamma(\alphaa +1)} dt  = \Kslice(q,q').$$
\end{theorem}

\begin{proof}
Thanks to the identity principle for slice regular functions, we need only to prove the result for a fixed $I\in \Sq$.  In fact,  the function $t\longmapsto \HolKernelFctR(t;z)$ belongs to $L^{2,\alpha}(\R^+,\C_I)$ and satisfies
$$\int_{0}^{+\infty}\HolKernelFctR(t;z)\overline{\HolKernelFctR(t;w)} \frac{t^\alphaa  e^{-t}}{\Gamma(\alphaa +1)} dt  = K^{\alpha}_{R,I}(z,w)$$
for every fixed $z,w\in \mathbb{B}_{R,I}$.
This follows readily using the generating function character of the kernel function $\HolKernelFctR(t;z)$
in  \eqref{cBergmanKernel}, to wit
$\HolKernelFctR(t;z)=\sum\limits_{n=0}^\infty z^n L^{(\alpha)}_n(t)$, combined with the orthogonality property \cite[Eq. (4), p. 205 - Eq. (7), p. 206]{Rainville71}
 \begin{align}\label{OrthLaguerre}
\int_{0}^{+\infty}L^{(\alphaa)}_{n}(t)L^{(\alphaa)}_{m}(t)\dfrac{t^{\alphaa }e^{-t}}{\Gamma(\alphaa +1)}dt
= \frac{\Gamma(\alphaa +n+1)}{\Gamma(n+1) \Gamma(\alphaa +1)} \delta_{n,m}.
\end{align}
Above $L^{(\alphaa)}_{n}(t)$ denotes the generalized Laguerre polynomials defined by \cite[p. 203 and P. 204]{Rainville71}
\begin{align}\label{LaguerrePoly}
 L^{(\alphaa)}_{n}(t)=\sum _{k=0}^{n} \frac{\Gamma(\alphaa +n+1)}{\Gamma(n-k+1)\Gamma(\alphaa +k+1)} \frac{(-t)^{k}}{k!}
= \frac{t^{-\alphaa }e^{t}}{n!}\frac{d^{n}}{dt^{n}}\left(t^{n+\alphaa } e^{-t}\right) .
\end{align}
\end{proof}

\begin{corollary}\label{SliceKernelFct}
For every fixed $q\in \BR $, the function $t\longmapsto \SHyperBKernelFctR(t;q)$ belongs to
$$L^{2,\alpha}(\R^+,\Hq) := L^2_{\Hq}\left(\R^+; \frac{t^{\alphaa }e^{-t}}{\Gamma (\alphaa +1)} dt\right) ; \, \alphaa >0,$$
 the right quaternionic Hilbert space of all square integrable $\Hq$-valued functions on the half-real line with respect to the scalar product
$$ \scal{\phi,\varphi}_{\alphaa ,\R^+} = \int_0^\infty \phi(t) \overline{\varphi(t)} \frac{t^\alpha e^{-t}}{\Gamma(\alphaa +1)}dt.$$
\end{corollary}

We conclude this section by giving an explicit closed formula of the kernel $\SHyperBKernelFctR$ in \eqref{KerFct2}.

\begin{theorem}\label{KerFctExplicit}
 For every $t \in \R^{+}$ and $q\in \BR $, we have
$$\SHyperBKernelFctR(t;q) = \widetilde{A}^\alpha_{R,slice} (t;q) :=  \left( 1 - \frac{q}{R} \right)^{-\alpha-1} \exp\left( \frac{tq}{q - R}\right) .$$
\end{theorem}

\begin{proof}
Fix $t\in \R^+$. It is clear that the restriction $\widetilde{A}^\alpha_{R,slice} (t;\cdot)$ to any $\BRI$ is holomorphic and coincides with the
kernel function $\HolKernelFctR(t;\cdot)$, given through \eqref{cBergmanKernel}, of the second Bargmann transform for the classical complex holomorphic Bergman space. On the other hand, the function $\SHyperBKernelFctR(t;q)$ is clearly slice regular in $q$-variable and coincides with $\HolKernelFctR(t;\cdot)$ when restricted to $\BRI$. Thus, by Lemma \ref{IdentityPrinciple}, we conclude that $\widetilde{A}^\alpha_{R,slice}(t;\cdot)=\SHyperBKernelFctR(t;\cdot)$ on the whole $\BR $.
\end{proof}

\section{An integral transform from $L^{2,\alpha}(\R^+,\Hq)$ onto $\BergSliceR$}
In this section, we consider and study a special integral transform from $L^{2,\alpha}(\R^+,\Hq)$ into $\BergSliceR$.
A complete orthonormal system for $L^{2,\alpha}(\R^+,\Hq) $ is given by 
 \begin{equation} \label{onbR}
  \phi_{n}(x) =\left( \frac{\Gamma(n+1)\Gamma(\alphaa+1)}{\Gamma(\alphaa +n+1)}\right)^{1/2} L ^{(\alphaa)}_{n}(x), \, n=0,1,\cdots,
\end{equation}
where $L^{(\alphaa)}_{n}(x)$ denotes the generalized Laguerre polynomials given by \eqref{LaguerrePoly}.
Recall also that the set of functions
 \begin{equation} \label{onbA}
f_n(q) =  \left(\frac{\Gamma(n+\alpha+1)}{n!\Gamma(\alpha+1)}\right)^{1/2} \left(\frac{q}{R} \right) ^{n},
\end{equation}
is an orthonormal basis of the slice hyperholomorphic Bergman space $\BergSliceR$ (see
Proposition \ref{prop:normBergman}).
Accordingly, by considering the kernel function $\SHyperBKernelFctR (t;q)$ defined on $\R^{+} \times \BR $ by \eqref{KerFct2} or equivalently by its explicit expression given in Theorem \ref{KerFctExplicit}, we can prove the following.

\begin{lemma}
The kernel function $\SHyperBKernelFctR$ can be realized as
$$\SHyperBKernelFctR (t;q) = \sum_{n=0}^\infty \phi_n(t) f_n(q) =
 \sum_{n=0}^\infty  \left(\frac{q}{R} \right) ^{n} L ^{(\alphaa)}_{n}(t)$$
for every $t \geq 0$ and $q\in \BR $.
\end{lemma}

\begin{proof}
This follows readily from the generating function \cite[Eq. (14), p. 135]{Rainville71}
 $$ \sum_{n=0}^\infty  \xi^{n} L^{(\alphaa)}_{n}(t) = \frac{1}{(1-\xi)^{\alphaa +1}} \exp\left(\frac{t\xi}{\xi - 1} \right).$$
Indeed, we have
 \begin{align*}  \sum_{n=0}^\infty \phi_n(t) f_n(q)&=
  \sum_{n=0}^\infty  \left(\frac{q}{R} \right) ^{n} L^{(\alphaa)}_{n}(t)=  \exp\left(\frac{tq}{q-R}\right) \left( 1-\frac{q}{R}\right)^{-\alphaa -1} .
  \end{align*}
 \end{proof}

Associated to $\SHyperBKernelFctR$, we perform the integral operator
 \begin{equation} \label{IntTransf}
 \SHyperBTransR \varphi(q)=  \left( 1-\frac{q}{R}\right)^{-\alphaa -1}  \int_{0}^{+\infty}   \exp\left(\frac{tq}{q-R}\right) \varphi(t) \frac{t^\alphaa  e^{-t}}{\Gamma(\alphaa +1)}dt.
 \end{equation}

 \begin{lemma}
The integral transform $\SHyperBTransR$ is well-defined on $ L^{2,\alpha}(\R^+,\Hq)$.
 \end{lemma}

 \begin{proof}
 This can be handled easily making use of the Cauchy-Schwarz inequality. Indeed, we have
\begin{align*}
|\SHyperBTransR \varphi(q)| &\leqslant \int_{0}^{\infty} |\SHyperBKernelFctR(t;q)| |\varphi(t)| \frac{t^\alphaa  e^{-t}}{\Gamma(\alphaa +1)} dt
\\&\leqslant \left( \int_{0}^{\infty} |\SHyperBKernelFctR(t;q)|^2\frac{t^\alphaa  e^{-t}}{\Gamma(\alphaa +1)} dt \right)^{\frac{1}{2}} \left( \int_{0}^{\infty}|\varphi(t)|^2 \frac{t^\alphaa  e^{-t}}{\Gamma(\alphaa +1)} dt \right) ^{\frac{1}{2}}.
\end{align*}
Now, since $t \longmapsto \SHyperBKernelFctR(t;q)$ belongs to $L^{2,\alpha}(\R^+,\Hq) $ for every fixed $q\in \Hq$ (see Theorem \ref{thm:Kernel-Kernel}), we deduce
\begin{align*}
|\SHyperBTransR \varphi(q)|
&\leqslant \norm{ \SHyperBKernelFctR(\cdot;q)}_{L^{2,\alpha}(\R^+,\Hq)}\norm{\varphi}_{L^{2,\alpha}(\R^+,\Hq)}.
\end{align*}
\end{proof}
Moreover, we can prove the following

\begin{theorem}\label{MainThm1}
The integral operator $\SHyperBTransR $
defines a unitary isometry transform from $L^{2,\alpha}(\R^+,\Hq)$ onto $\BergSliceR.$ Moreover, we have
$\SHyperBTransR \phi_{n}(q) = f_n(q),$ where $\phi_{n}(x)$ and $f_n(q)$ are respectively the orthogonal bases of $L^{2,\alpha}(\R^+,\Hq)$ and the slice hyperholomorphic Bergman space $\BergSliceR$  given by \eqref{onbR} and \eqref{onbA}, respectively.
\end{theorem}

\begin{proof}
Fix $I\in \Sq$. The identification of the slice $\C_I$ with the complex plane $\C$ and $\BRI$ with the disc $D_R=D(0,R)$ of $\C$ leads to the consideration of the unitary isometry $\HolBTransR$ from $L^{2,\alpha}(\R^+,\C_I)$ onto $\BergHolRI$. It is specified by the rescaled version of the formula \eqref{sBT}, to wit
\begin{align*}
\HolBTransR \varphi(z) = \left(1-\frac{z}{R}\right)^{-\alpha-1} \int_0^{+\infty}  \varphi(t) \exp \left( \frac{tz}{z-R}\right)
 \frac{t^{\alpha}e^{-t}}{\Gamma (\alpha+1)}   dt
\end{align*}
for $z,w\in \BRI$. Now, for $J\in \Sq$ such that $J\perp I$, we split any $\varphi\in L^{2,\alpha}(\R^+,\Hq)$ as $\varphi_I = F + G J$ for some $F,G: \R^+_I \longrightarrow \C_I$. Obviously, we have $\R^+_I = \R^+$, $\varphi_I=\varphi$ as well as
$$ (\SHyperBTransR \varphi)_I = \SHyperBTransR F + (\SHyperBTransR G) J = \HolBTransR F + (\HolBTransR G) J$$
by means of \eqref{IntTransf} and therefore $\SHyperBTransR \phi_{n}(q) = f_n(q)$.
Moreover, it is evident to see that $F,G \in L^{2,\alpha}(\R^+,\C_I)$ with
\begin{align*}
 \norm{\SHyperBTransR \varphi}_{\alpha,slice}^2  &= \norm{ \HolBTrans F}_{\alpha,slice}^2 + \norm{ \HolBTrans G}_{\alpha,slice}^2\\
&= \norm{F}_{L^{2,\alpha}(\R^+,\C_I)}^2 + \norm{G}_{L^{2,\alpha}(\R^+,\C_I)}^2\\
&=\norm{\varphi}_{L^{2,\alpha}(\R^+,\Hq)}^2 .
\end{align*}
Accordingly, the transform  $\SHyperBTransR$ from $L^{2,\alpha}(\R^+,\Hq)$ into $\BergSliceR$ is injective and an isometry.
On the other hand, since $\HolBTransR$ is surjective, we see that so is $\SHyperBTransR$. Therefore, $\SHyperBTransR$ is a unitary isometry from $L^{2,\alpha}(\R^+,\Hq)$ onto $\BergSliceR$.
\end{proof}

\begin{remark}
The argument of splitting $\varphi$ as $\varphi_I = F + G J$ for some $F,G: \R^+_I \longrightarrow \C_I$ with $J\in \Sq$ and that $J\perp I$ is used in \cite{RenWang2016} and is the basic idea in the splitting Lemma \ref{split}. The result of Theorem \ref{MainThm1} can also be proved by using \eqref{KerFct2} in order to rewrite the integral transform $\SHyperBTransR$ acting on $L^{2,\alpha}(\R^+,\Hq)$, in \eqref{IntTransf}, as
\begin{align*}
\SHyperBTransR  \varphi(q)
= \frac{1 - I_{q}J}{2}
& \int_{0}^{+\infty}\HolKernelFctR(t;z_{q})  \varphi(t) \frac{t^\alphaa  e^{-t}}{\Gamma(\alphaa +1)} dt
\\& +  \frac{1 + I_{q}J}{2}\int_{0}^{+\infty}\HolKernelFctR(t;\overline{z_{q}})  \varphi(t) \frac{t^\alphaa  e^{-t}}{\Gamma(\alphaa +1)} dt.
\end{align*}
\end{remark}

The second main result of this section is the following

\begin{theorem}\label{MainThm2}
The inverse transform $ [\SHyperBTransR]^{-1}: \BergSliceR\longrightarrow L^{2,\alpha}(\R^+,\Hq)$ of $\SHyperBTransR$ is given by
\begin{align}\label{Inverse2BTq}
[\SHyperBTransR]^{-1} f(t) =  \left(\frac{\alpha}{\pi R^2}\right)  \int_{\BRI}
 \exp \left( \frac{t\overline{q}}{\overline{q}-R}\right) \frac{\left(1- \frac{|q|^2}{R^2} \right)^{\alpha-1}}{ \left(1-\frac{\overline{q}}{R}\right)^{\alphaa+1}}   f(q) dxdy.
\end{align}
\end{theorem}

\begin{proof}
The inverse of the unitary isometric transform $\HolBTransR$ in \eqref{IntTransf} is given by its hermitian conjugate (the adjoint).
More explicitly, we claim that
\begin{align}\label{Inverse2BTc}
[\HolBTransR]^{-1} F(t) = \left(\frac{\alphaa}{\pi R^2}\right)  \int_{\BRI}
F(z) \exp \left( \frac{t\overline{z}}{\overline{z}-R}\right)  \frac{\left(1- \frac{|z|^2}{R^2}\right)^{\alpha-1}}{(1-\frac{\overline{z}}{R})^{\alpha+1}}  dxdy
\end{align}
for every $F\in \BergHolRI $. Consequently, the inverse transform of $\SHyperBTransR$ in \eqref{IntTransf}, the quaternionic analogue of $\HolBTransR$, is given by \eqref{Inverse2BTq}.
This follows readily from the splitting Lemma \ref{split} combined with \eqref{Inverse2BTc} above.
 \end{proof}

\section{\bf Asymptotic behavior: from slice Bergman to slice Bargmann}

Intuitively, the space $\Hq$ can be viewed as the euclidean limit of balls $\BR$ in $\Hq$, as the radius $R$ goes to $+\infty$.
This intuitive limit can be justified geometrically.
The analogue of the Poincar\'e (the real hyperbolic) differential metric on the unit open ball $\B$ in the quaternionic setting
is given by $ ds^2_{\B} = (1-|q|^2)^{-2}|d_Iq|^2 $, where $q= x+Iy$; $I\in \Sq$, and $d_Iq=dx+Idy$. It was proposed by Bisi and Gentili in \cite{BisiGentili2009} by developing a variation of an approach adopted by Ahlfors \cite{Ahlfors1988}.
The quaternionic hyperbolic geometry on $\BR$ is described by the scaled Poincar\'e-type differential metric
$$ ds^2_{\BR} = R^4 (R ^2-|q|^2)^{-2} |d_Iq|^2.  $$
The associated volume measure is given by
$$d\lambda_{I,R}^{\alpha}(q=x+Iy) =\left(\frac{\alpha}{\pi R^2}\right)\left(1-\frac{|q|^2}{R^2}\right)^{\alpha-1} dxdy .$$
Therefore, the sectional curvature of $(\BR,ds^2_{\BR})$, given by $\kappa_R= - 4/R^2$, tends to $0$ which corresponds to the curvature of the flat hermitian manifold $(\Hq,ds^2_{\Hq})$ endowed with the flat metric $ ds^2_{\Hq} = |d_Iq|^2 $.
Moreover, if we parameterize $\alphaa$ as $\alphaa = \nu R^2$ for some fixed $\nu>0$, we see that the sliced measure $d\lambda_{I,R}^{\alpha}$
converges pointwisely to the sliced volume measure on $\Hq$,
$$ \displaystyle \lim\limits_{R\rightarrow +\infty} d\lambda_{I,R}^{\nu R^2}(q=x+Iy) = \left(\frac{\nu}{\pi}\right) \lim_{R\rightarrow +\infty} \left( 1- \left|\dfrac{q}R\right|^2\right) ^{\nu R^2}  dxdy  = \left(\frac{\nu}{\pi}\right) e^{-\nu |q|^{2}}  dxdy .$$

With respect to this parametrization, the orthonormal basis of the slice hyperholomorphic Bergman space $\BergSliceR$ given by the functions
 \begin{equation*}
f_n(q) =  \left(\frac{\Gamma(n+\nu R^2+1)}{n!\Gamma(\nu R^2+1)}\right)^{1/2} \left(\frac{q}{R} \right) ^{n},
\end{equation*}
 (see Proposition \ref{prop:normBergman}), also gives rise to
  \begin{equation*}
e_n(q) =  \left(\frac{\nu^n }{n!}\right)^{1/2} q^{n},
\end{equation*}
pointwisely, when $R$ goes to infinity. The set of $e_n(q)$ is in fact an orthonormal basis of the slice hyperholomorphic Bargmann-Fock space
\begin{equation}\label{SliceBargSp}
 \mathcal{F}^{2,\nu}_{slice}(\Hq) = \mathcal{SR}(\Hq) \cap L^{2,\nu}(\C_I,\Hq )
 \end{equation} 
 with respect to the sliced Gaussian measure $\left(\frac{\nu}{\pi}\right) e^{-\nu |q|^{2}}  dxdy $. This follows readily thanks to the Binet formula \cite{Rainville71}
$$ \lim_{x \to +\infty} \frac{\Gamma(x+a)}{x^{a-b}\Gamma(x+b)} =  1 .$$
The main result of this section concerns the pointwise convergence of the reproducing kernel function.

\begin{theorem}\label{MainThm3:limKernel}
The pointwise limit of the weighted Bergman kernel $K^{\nu R^2}_{R,slice}$ of the slice hyperholomorphic Bergman space $A^{2,\nu R^2}_{slice}(\BR)$
is exactly the reproducing kernel of the slice hyperholomorphic Bargmann-Fock space $ \mathcal{F}^{2,\nu}_{slice}(\Hq) $ in \eqref{SliceBargSp}.
More precisely, for every fixed $(q,p) \in \Hq\times \Hq$, we have
$$ \lim_{R\rightarrow +\infty}  K^{\nu R^2}_{R,slice}(q, p) = e_{*}^{[\nu q,\overline{p}]} :=\sum_{n=0}^{+\infty}\dfrac{\nu^n q^{n}\overline{p}^{n}}{n!}.$$
\end{theorem}

\begin{proof}
Recall first that for $\alphaa=\nu R^2$ and $R$ being reals, the expression of the reproducing kernel $\Kslice$ given by \eqref{explicitWBK2} reads
\begin{align*}
\Kslice(q, p)=  {_2{F^{*}}_1}\left(   \begin{array}{c} \nu R^2+1, 0 \\ 0 \end{array} \bigg | \left[\frac{q}{R^2} , \overline{p}\right] \right).
\end{align*}
 Accordingly, what is needed to conclude is an asymptotic behavior of the involved hypergeometric function.
Thus, we claim that for every fixed $q,p\in \Hq$ and reals $a,b,c$, we have (see \cite{GhIn2005JMP} for a rigorous proof that we can extend to our context by means of the identity principle for slice regular functions)
\begin{align}\label{limitHyepergeometric}
 \lim\limits_{\rho \longrightarrow + \infty}
 {_2{F^{*}}_1}\left( \begin{array}{c} a + \rho  , b \\ c \end{array}\bigg | \left[\frac{q}{\rho},\overline{p}\right] \right)=
 {_1{F^*}_1}\left( \begin{array}{c} b \\ c \end{array}\bigg | \left[q,\overline{p}\right] \right).
\end{align}
Moreover, the convergence is uniform on compact sets of $\Hq\times \Hq$. Therefore, one obtains
$$ \lim_{R\rightarrow +\infty}  K^{\nu R^2}_{R,slice}(q, p) =
= \sum_{n=0}^{+\infty}\dfrac{\nu^n q^{n}\overline{p}^{n}}{n!} = e_{*}^{[\nu q,\overline{p}]} .$$
\end{proof}

  \begin{remark}
 According to the uniform convergence of the series in \eqref{limitHyepergeometric} on compact sets of $\Hq\times \Hq$, the convergence in Theorem  \ref{MainThm3:limKernel} of the reproducing kernel function is uniform in $(q,p)$ in any compact set of $\Hq\times \Hq$.
\end{remark}

\section{Appendix}

The suggested expression of the weighted Bergman kernel $\Kslice(q, p)$ in \eqref{explicitWBK} can be handled as follows (which is in fact a different proof of Theorem \ref{thm:explicitWBK}).
 Fix $p\in \BR $ and let $q=x+Iy\in \BR $. Take $z_p$ to be $z_p=x+I_py$. Then, by the representation formula and the explicit expression of the classical weighted Bergman kernel given through \eqref{expKhol}, we obtain
\begin{align*}
\Kslice(q, p) &= \Kslice(x+Iy, p)
\\&= \frac{1}{2} \left( \Kslice(z_p, p) + \Kslice(\overline{z_p}, p)\right) + \frac{II_p}{2} \left( \Kslice(\overline{z_p}, p) - \Kslice(z_p, p) \right)\\
&= \frac{1}{2} \left( \frac{\left( 1 - \frac{z_p\overline{p}}{R^2} \right)^{\alphaa+1} + \left( 1 - \frac{\overline{z_p}\overline{p}}{R^2} \right)^{\alphaa+1}}{\left( 1 - \frac{z_p\overline{p}}{R^2} \right)^{\alphaa+1}  \left( 1 - \frac{\overline{z_p}\overline{p}}{R^2} \right)^{\alphaa+1}} \right)
 + \frac{II_p}{2} \left( \frac{\left( 1 - \frac{z_p\overline{p}}{R^2} \right)^{\alphaa+1} - \left( 1 - \frac{\overline{z_p}\overline{p}}{R^2} \right)^{\alphaa+1}}{\left( 1 - \frac{z_p\overline{p}}{R^2} \right)^{\alphaa+1}  \left( 1 - \frac{\overline{z_p}\overline{p}}{R^2} \right)^{\alphaa+1}} \right).
\end{align*}
Straightforward computation shows that
$$\left( 1 - \frac{z_p\overline{p}}{R^2} \right)^{\alphaa+1}  \left( 1 - \frac{\overline{z_p}\overline{p}}{R^2} \right)^{\alphaa+1} =
\left( 1 - 2 \frac{\Re(q) \overline{p} }{R^2} +  \frac{|q|^2\overline{p}^2}{R^4} \right)^{\alphaa+1}$$
and
$$\left( 1 - \frac{z_p\overline{p}}{R^2} \right)^{\alphaa+1}  \pm \left( 1 - \frac{\overline{z_p}\overline{p}}{R^2} \right)^{\alpha+1} =
\sum_{k=0}^\infty \frac{(-\alpha-1)_k}{k!} (z_p^k \pm \overline{z_p}^k) \left(\frac{\overline{p}}{R^2}\right)^k.$$
The equality follows using the binomial theorem for real exponent
$$ (a - b)^\beta = \sum_{k=0}^\infty \frac{(-\beta)_k}{k!} a^{\beta-k} b^k$$
valid for $|a| > |b|$.
Consequently, the expression of $\Kslice(q, p)$ becomes
\begin{align*}
\Kslice(q, p)
&=  \left(\sum_{k=0}^\infty   \frac{(-\alpha-1)_k}{k!}\frac{1}{2} \left[ (z_p^k + \overline{z_p}^k) + I I_p (z_p^k - \overline{z_p}^k) \right]\left(\frac{\overline{p}}{R^2}\right)^k \right)\\
& \qquad \times
\left( 1 - 2 \frac{\Re(q) \overline{p} }{R^2} +  \frac{|q|^2\overline{p}^2}{R^4} \right)^{-\alpha-1}.
\end{align*}
Now, making use of the well-established fact
$$\overline{q}^{k}=(x-Iy)^{k}=\frac{1}{2}\left( ((x+Jy)^{k}+(x-Jy)^{k})  + IJ((x+Jy)^{k} - (x-Jy)^{k}) \right),$$
for every nonnegative integer $k$ and arbitrary $q\in\Hq$ and $J\in\Sq$, we conclude easily that
\begin{align*}
\Kslice(q, p) &=  \left(\sum_{k=0}^\infty    \frac{ (-\alphaa-1)_k  }{k!} \frac{ \overline{q}^k \overline{p}^k}{R^{2k}} \right)
\left( 1 - 2 \frac{\Re(q) \overline{p} }{R^2} +  \frac{|q|^2\overline{p}^2}{R^4} \right)^{-\alphaa-1}.
\end{align*}

{\bf\it Acknowledgements.}
The authors acknowledge their indebtedness to the anonymous referees for helpful comments and for brought to our attention the recent work  \cite{RenWang2016}.
 A part of the present investigation was completed while the second author visited to departimenti di Mathematica at the Politecnico di Milano for May - June 2017.  He wishes to express his gratitude to Professor I.M. Sabadini for hospitality and many interesting stimulating discussions and interactions.
This research work was supported, in part, by a grant from the Simons Foundation. 

\end{document}